\documentclass[11pt,a4paper]{amsart}
\usepackage[british]{babel}

\usepackage{a4wide}
\usepackage{amssymb}
\usepackage{amsmath}
\usepackage{amsthm}
\usepackage{array}
\usepackage[hidelinks]{hyperref}
\usepackage{breakurl}
\usepackage[maxnames=4,backend=bibtex,style=alphabetic]{biblatex}
\newtheorem{thm}{Theorem}[section]

\newtheorem{cor}[thm]{Corollary}
\newtheorem{lem}[thm]{Lemma}

\newtheorem{prop}[thm]{Proposition}

\theoremstyle{definition}

\newtheorem{exmpl}[thm]{Example}

\newtheorem{definition}[thm]{Definition}

\newtheorem{remark}[thm]{Remark}

\renewcommand{\epsilon}{\varepsilon}
\renewcommand{\phi}{\varphi}
\newcommand{\defeq}{\mathrel{\mathop:}=}

\DeclareMathOperator{\id}{id}
\DeclareMathOperator{\ar}{ar}
\DeclareMathOperator{\h}{ht}

\begin{document}


\title{Profinite algebras and affine boundedness}

\author{Friedrich Martin Schneider}
\author{Jens Zumbr\"agel}
\address{Institute of Algebra, TU Dresden, 01062 Dresden, Germany}

\date{\today}

\begin{abstract} We prove a characterization of profinite algebras, i.e., topological algebras that are isomorphic to a projective limit of finite discrete algebras.  In general profiniteness concerns both the topological and algebraic characteristics of a topological algebra, whereas for topological groups, rings, semigroups, and distributive lattices, profiniteness turns out to be a purely topological property as it is is equivalent to the underlying topological space being a Stone space.

Condensing the core idea of those classical results, we introduce the concept of affine boundedness for an arbitrary universal algebra and show that for an affinely bounded topological algebra over a compact signature profiniteness is equivalent to the underlying topological space being a Stone space.  Since groups, semigroups, rings, and distributive lattices are indeed affinely bounded algebras over finite signatures, all these known cases arise as special instances of our result. Furthermore, we present some additional applications concerning topological semirings and their modules, as well as distributive associative algebras. We also deduce that any affinely bounded simple compact algebra over a compact signature is either connected or finite. Towards proving the main result, we also establish that any topological algebra is profinite if and only if its underlying space is a Stone space and its translation monoid is equicontinuous. \end{abstract}

\maketitle




\section{Introduction}

A topological algebra is called profinite if it is representable as a projective limit of finite discrete algebras. Profiniteness is a property referring to the interplay between the topological and the algebraic structure of a topological algebra rather than a simple conjunction of topological and algebraic conditions. However, for topological groups profiniteness turns out to be a purely topological phenomenon as it is equivalent to the underlying topological space being a \emph{Stone space}, i.e., a totally disconnected compact Hausdorff space. This is due to a classical result by van Dantzig \cite{VanDantzig}. Moreover, the same happens to be true for topological rings due to Anzai \cite{anzai} (see also \cite{kaplansky-rings}) as well as for topological semigroups and distributive lattices according to Numakura \cite{Numakura57}. In fact, Anzai \cite{anzai} even proved that a topological ring is profinite if and only the underlying topological space is a compact Hausdorff space. Since fields constitute simple rings, this furthermore implies that the only compact Hausdorff topological fields are the finite discrete ones. To our knowledge, it has been an open question to classify those classes of topological algebras for which profiniteness is equivalent to the underlying space being a Stone space, cf.~\cite{banaschewski-pua,choe,bergman,day79,johnstone,clark_et_al}.

In the present paper we condense a common core idea from the proofs of the results mentioned above. In fact, we introduce the concept of \emph{affine boundedness} for an arbitrary abstract algebra (Definition~\ref{definition:polynomial.boundedness}), and show that for an affinely bounded topological algebra over a compact signature profiniteness is equivalent to the underlying topological space being a Stone space (Theorem~\ref{theorem:bounded.profinite.algebras}). Since groups, semigroups, rings, and distributive lattices are indeed affinely bounded algebras over finite signatures, all the results addressed above arise as special instances of Theorem~\ref{theorem:bounded.profinite.algebras}. As an additional application we obtain a corresponding characterization of profinite topological semirings and profinite modules over compact topological semirings. Moreover, we provide a new conceptual proof of a profiniteness result by Choe~\cite{choe} for distributive associative algebras. As another corollary of Theorem~\ref{theorem:bounded.profinite.algebras} we obtain a dichotomy for affinely bounded simple compact Hausdorff topological algebras over compact signatures -- they are either connected or finite (Corollary~\ref{corollary:dichotomy}). Along the way towards our main result, we show that a general topological algebra is profinite if and only if its underlying space is a Stone space and its translation monoid is equicontinuous (Theorem~\ref{theorem:first.main.theorem}), or, equivalently, relatively compact with respect to the compact-open topology (Corollary~\ref{corollary:first.main.theorem}).

We note that our paper is related to work of Clark, Davey, Freese, and Jackson~\cite{clark_et_al} (see also~\cite{clark_et_al_2}).  In this article, the authors consider algebras having finitely determined syntactic congruences (FDSC), which means that there is a finite set of terms that determines all congruences of the algebra, and they show that such an algebra is profinite if and only if its underlying topological space is a Stone space (\cite{clark_et_al}, Theorem~8.1).  In our approach, instead of considering congruences directly, we focus on the translation monoid of the algebra.  Indeed, our concept of affine boundedness is associated with the translation monoid, and the characterization of profiniteness in Theorem~\ref{theorem:first.main.theorem} is based on topological properties of the translation monoid.  Furthermore, our setup allows for compact signatures (rather than just finite ones) and therefore applies -- for instance -- to modules over compact topological semirings, whereas the concept of FDSC is not suited for this situation.

This article is organized as follows. In Section~\ref{section:polynomial.boundedness} we introduce and explore our concept of affine boundedness: after recalling some standard terminology from universal algebra, we give the definition of affine boundedness for a general algebra, investigate several examples, and deduce some useful consequences concerning the representation of an affinely bounded algebra's translation monoid. Based on that, in Section~\ref{section:polynomially.bounded.topological.algebras} we show that any affinely bounded compact topological algebra over a compact signature admits a compact translation monoid. The subsequent Section~\ref{section:profinite.topological.algebras} is devoted to studying profinite topological algebras. In the course of this, we establish the aforementioned characterization of profiniteness in terms of the translation monoid of a topological algebra. Utilizing the results of Section~\ref{section:polynomially.bounded.topological.algebras}, we conclude that an affinely bounded topological algebra over a compact signature is profinite if and only if its carrier space happens to be a Stone space. In Section~\ref{section:simple.topological.algebras} we deduce the above-mentioned topological dichotomy for affinely bounded simple compact algebras over compact signatures. For illustration purposes we present some examples of affinely unbounded topological algebras in Section~\ref{section:examples}. Finally, Section~\ref{section:applications} provides applications of our results to profinite topological semirings and profinite modules over compact topological semirings, as well as distributive associative profinite topological algebras in the sense of Choe~\cite{choe}.

\section{Affine boundedness}\label{section:polynomial.boundedness}

In this section we introduce the concept of \emph{affine boundedness} for general algebras. This property refers to the representation of an algebra's translation monoid in terms of linear polynomials. For a start, let us recall some basic terminology from universal algebra. More precisely, we shall agree on some notation concerning terms and polynomials. For background information we refer to \cite{burris}.

Throughout this section, let $\Omega = (\Omega_{n})_{n \in \mathbb{N}}$ be a \emph{signature}, i.e., a sequence of disjoint sets. Furthermore, let us fix a countably infinite set $X$ of variables. We denote by $T_{\Omega}(X)$ the set of \emph{$\Omega$-terms over $X$}, i.e., the smallest set $T$ subject to the following conditions: 
\begin{enumerate}
	\item $X \subseteq T$.
	\item $\omega t_{1} \ldots t_{n} \in T$ for all $n \in
	\mathbb{N}$, $\omega \in \Omega_{n}$ and $t_{1},\ldots,t_{n} \in T$.
\end{enumerate} Let us abbreviate $T_{\Omega}^{\times}(X) \defeq T_{\Omega}(X)\setminus T_{\Omega}(\varnothing)$. We define the \emph{height function} $\h \colon T_{\Omega}(X) \to \mathbb{N}$ in the following recursive manner: $\h(x) \defeq 0$ whenever $x \in X \cup \Omega_{0}$ and \begin{displaymath}
	\h(\omega t_{1} \ldots t_{n}) \defeq \sup \{ \h(t_{i}) \mid i \in \{ 1,\ldots,n \} \} + 1
\end{displaymath} for all $n \in \mathbb{N}\setminus \{ 0 \}$, $\omega \in \Omega_{n}$ and $t_{1},\ldots,t_{n} \in T_{\Omega}(X)$. Another recursive definition provides us with the \emph{arity} of a term, i.e., the map $\ar \colon T_{\Omega}(X) \to \mathbb{N}$ defined as follows: $\ar(x) \defeq 0$ whenever $x \in X$ and \begin{displaymath}
	\ar(\omega t_{1} \ldots t_{n}) \defeq \sup (\{ \ar(t_{i}) \mid i \in \{ 1,\ldots,n \} \} \cup \{ n \})
\end{displaymath} for all $n \in \mathbb{N}$, $\omega \in \Omega_{n}$ and $t_{1},\ldots,t_{n} \in T_{\Omega}(X)$.

In the following, we shall have a closer look at a very particular subclass of $T_{\Omega}(X)$. A term $t \in T_{\Omega}(X)$ is called \emph{linear} if each variable appearing in $t$ occurs exactly once in $t$. We denote by $L_{\Omega}(X)$ the set of all linear $\Omega$-terms over $X$. Note that $L_{\Omega}(\varnothing) = T_{\Omega}(\varnothing)$. Again, we abbreviate $L_{\Omega}^{\times}(X) \defeq L_{\Omega}(X)\setminus L_{\Omega}(\varnothing)$. Concerning a single variable $x$, we obtain a monoid by equipping $L_{\Omega}(x) \defeq L_{\Omega}(\{ x \})$ with the obvious concatenation along $x$, i.e., the operation ${\cdot} \colon L_{\Omega}(x) \times L_{\Omega}(x) \to L_{\Omega}(x)$ defined recursively as follows: if $t \in L_{\Omega}(x)$, then $x \cdot t \defeq t$, and $(\omega t_{1}\ldots t_{n}) \cdot t \defeq \omega (t_{1} \cdot t) \ldots (t_{n} \cdot t)$ for all $n \in \mathbb{N}$, $\omega \in \Omega_{n}$ and $t_{1},\ldots ,t_{n} \in L_{\Omega}(x)$. Note that $L_{\Omega}^{\times}(x) \defeq L_{\Omega}(x)\setminus L_{\Omega}(\varnothing)$ constitutes a submonoid of $L_{\Omega}(x)$.

In the course of this article, we shall also be concerned with polynomials. To recall this concept, let us consider a set $A$ such that $A \cap \Omega_{n} = \varnothing$ for each $n \in \mathbb{N}$. We define $\Omega + A$ to be the signature where $(\Omega + A)_{0} \defeq \Omega_{0} \cup A$ and $(\Omega + A)_{n} \defeq \Omega_{n}$ for all $n \geq 1$. The elements of $T_{\Omega + A}(X)$ are called \emph{$\Omega$-polynomials over $X$ and $A$}. In particular, the elements of $L_{\Omega + A}(x)$ are the linear $\Omega$-polynomials over~$A$, which we may also call {\em affine} $\Omega$-terms over~$A$.

Now we come to algebras. By an \emph{$\Omega$-algebra} or \emph{algebra of type $\Omega$} we mean a pair $\mathbf{A} = (A,E)$ consisting of a set $A$ and a family $E = (E_{n})_{n \in \mathbb{N}}$ of maps $E_{n} \colon \Omega_{n} \times A^{n} \to A$ ($n \in \mathbb{N}$). Suppose $\mathbf{A} = (A,E)$ to be an $\Omega$-algebra. For $n \in \mathbb{N}$ and $\omega \in \Omega_{n}$, we define $\omega^{\mathbf{A}} \colon A^{n} \to A, \, a \mapsto E_{n}(\omega,a)$. A \emph{translation} of $\mathbf{A}$ is a map of the form \begin{displaymath}
A \to A, \quad x 
\mapsto E_{n}(\omega,a_{1},\ldots,a_{i-1},x,a_{i+1},\ldots,a_{n})
\end{displaymath} where $n \in \mathbb{N} \setminus \{ 0 \}$, $\omega \in \Omega_{n}$, $i \in \{1,\ldots,n\}$ and $a_{1},\ldots,a_{i-1},a_{i+1},\ldots,a_{n} \in A$. We denote by $M(\mathbf{A})$ the \emph{translation monoid} of~$\mathbf{A}$, i.e., the transformation monoid generated by the translations of $\mathbf{A}$.

Let us recall the following well-known fact:

\begin{lem}\label{lemma:congruences} Let $\mathbf{A} = (A,E)$ be an algebra of type $\Omega$. An equivalence relation $\theta$ on $A$ is a congruence of $\mathbf{A}$ if and only if $(f(a),f(b)) \in \theta$ for all $(a,b) \in \theta$ and $f \in M(\mathbf{A})$. \end{lem}

Since the congruences of an algebra are precisely its translation invariant equivalence relations, we aim at a feasible description of the corresponding translation monoid. To this end, we define a map $\psi_{\mathbf{A}} \colon T_{\Omega + A}(x) \times A \to A$ recursively as follows: if $z \in A$, then $\psi_{\mathbf{A}}(x,z) \defeq z$ and $\psi_{\mathbf{A}}(a,z) \defeq a$ for $a \in A$, as well as \begin{displaymath}
	\psi_{\mathbf{A}}(\omega t_{1}\ldots t_{n},z) \defeq E_{n}(\omega,\psi_{\mathbf{A}}(t_{1},z), \ldots, \psi_{\mathbf{A}}(t_{n},z))
\end{displaymath} for $n \in \mathbb{N}$, $\omega \in \Omega_{n}$ and $t_{1},\ldots,t_{n} \in T_{\Omega}(X)$. Furthermore, we shall consider the map $\Psi_{\mathbf{A}} \colon T_{\Omega + A}(x) \to A^{A}$ defined by $\Psi_{\mathbf{A}}(t)(z) = \psi_{\mathbf{A}}(t,z)$ for all $t \in T_{\Omega + A}(x)$ and $z \in A$.

A straightforward term induction reveals the subsequent result:

\begin{prop} If $\mathbf{A} = (A,E)$ is an $\Omega$-algebra, then $M(\mathbf{A}) = \{ \Psi_{\mathbf{A}}(t) \mid t \in L_{\Omega + A}^{\times}(x) \}$. \end{prop}

Note that $\Psi_{\mathbf{A}} \colon (L_{\Omega+A}^{\times}(x), \cdot) \to (M(\mathbf{A}), \circ)$ is a surjective monoid homomorphism.
Now we introduce and explain the aforementioned concept of boundedness.

\begin{definition}\label{definition:polynomial.boundedness} Let $\mathbf{A} = (A,E)$ be an algebra of type $\Omega$. We say that $\mathbf{A}$
is \emph{affinely bounded by $m \in \mathbb{N}$} if \begin{displaymath}
	M(\mathbf{A}) = \{ \Psi_{\mathbf{A}}(t) \mid t \in L_{\Omega + A}^{\times}(x), \, \ar(t) \leq m, \, \h(t) \leq m \} .
\end{displaymath} We call $\mathbf{A}$ \emph{affinely bounded} if $\mathbf{A}$ is affinely bounded by some $m \in \mathbb{N}$. \end{definition}

\begin{remark}\label{remark:polynomial.boundedness} Let $\mathbf{A} = (A,E)$ be an $\Omega$-algebra and let $m \in \mathbb{N}$. The following are equivalent: \begin{enumerate}
	\item $\mathbf{A}$ is affinely bounded by $m$.
	\item $\forall t \in L_{\Omega + A}(x) \, \exists t' \in L_{\Omega + A}(x) \colon \, \Psi_{\mathbf{A}}(t) = \Psi_{\mathbf{A}} (t'), \, \ar(t') \leq m, \, \h(t') \leq m$.
\end{enumerate} Moreover, if $\sup \{ n \in \mathbb{N} \mid \Omega_{n} \ne \varnothing \} \leq m$, then the following are also equivalent to (1): \begin{enumerate}
	\item[(3)] $M(\mathbf{A}) = \{ \Psi_{\mathbf{A}}(t) \mid t \in L_{\Omega + A}^{\times}(x), \, \h(t) \leq m \}$.
	\item[(4)] $\forall t \in L_{\Omega + A}(x) \, \exists t' \in L_{\Omega + A}(x) \colon \, \Psi_{\mathbf{A}}(t) = \Psi_{\mathbf{A}} (t'), \, \h(t') \leq m$.
\end{enumerate} \end{remark}

\begin{exmpl}\label{example:polynomial.boundedness} Let $\mathbf{G} = (G, \cdot)$ be a semigroup, i.e., $\mathbf{G}$ is an $\Omega$-algebra, where $\Omega_2 = \{ \cdot \}$ and $\Omega_i = \varnothing$ for $i \ne 2$, and the binary operation~$\cdot$ is associative.  Then the $\Omega$-algebra $\mathbf{G}$ is affinely bounded by~$2$, that is, for each $t \in L_{\Omega + G}(x)$ there exists $t' \in L_{\Omega + G}(x)$ with $\Psi_{\mathbf{G}}(t) = \Psi_{\mathbf{G}}(t')$ and $\h(t') \le 2$.  Indeed, due to associativity we can choose $t' = (a \cdot x) \cdot b$ for some $a, b \in G$, or $t' = a \cdot x$ or $t' = x \cdot a$ for some $a \in G$, or $t' = x$.

Similarly, any monoid $\mathbf{G} = (G, \cdot, 1)$ is affinely bounded (by~$2$), as the binary operation~$\cdot$ is associative.  If $\mathbf{G} = (G, \cdot, 1, {}^{-1})$ is a group then the algebra~$\mathbf{G}$ is likewise affinely bounded, in fact, for each $t \in L_{\Omega + G}(x)$ there exists $t' \in L_{\Omega + G}(x)$ with $\h(t') \le 3$ and $\Psi_{\mathbf{G}}(t) = \Psi_{\mathbf{G}}(t')$, namely $t' = (a \cdot x) \cdot b$ or $t' = (a \cdot x^{-1}) \cdot b$, for some $a, b \in G$.

On the other hand, the $\Omega$-algebra $\mathbf{G} = (G, \cdot)$ is in general not affinely bounded if the binary operation~$\cdot$ is non-associative.  For example, let $\mathbf{G}$ be the free groupoid over one element~$a$, and consider the sequence $(t_i)_{i \in \mathbb{N}}$ recursively defined by $t_0 := x$ and $t_i := t_{i-1} \cdot a$ for $i \ge 1$ (e.g., $t_3 = ((x \cdot a) \cdot a) \cdot a$), then for each $i \in \mathbb{N}$ there is no $t' \in L_{\Omega + G}(x)$ with $\Psi_{\mathbf{G}}(t_i) = \Psi_{\mathbf{G}}(t')$ and $\h(t') < i$. \end{exmpl}

Affine boundedness offers a very convenient way of describing an algebra's translation monoid (see Lemma~\ref{lemma:bounded.translation.monoid}), which shall turn out useful in Section~\ref{section:polynomially.bounded.topological.algebras}. Towards this aim, we need to address some technical matters. We consider the signature $\Sigma \defeq (\Sigma_{n})_{n \in \mathbb{N}}$ where $\Sigma_{0} \defeq \{ 0, \ast \}$ and $\Sigma_{n} \defeq \{ n \}$ for every $n \geq 1$. Let $\mathbf{A}$ be an algebra of type $\Omega$. For each $t \in T_{\Sigma}(x)$, we define a set $S(t)$ and a map $\phi_{\mathbf{A}}(t)\colon S(t) \times A \to A$ recursively as follows. Let $S(x) \defeq \{ \varnothing \}$ and $S(\ast) \defeq A$, as well as \begin{displaymath} 
	S (n t_{1} \ldots t_{n}) \defeq \Omega_{n} \times S(t_{1}) \times \ldots \times S(t_{n})
\end{displaymath} for $n \in \mathbb{N}$ and $t_{1},\ldots,t_{n} \in T_{\Sigma}(x)$. Besides, let $\phi_{\mathbf{A}}(x) \colon \{ \varnothing \} \times A \to A , \, (\varnothing,z) \mapsto z$ and $\phi_{\mathbf{A}}(\ast ) \colon A \times A \to A, \, (a,z) \mapsto a$, as well as \begin{multline*} 
	\quad \phi_{\mathbf{A}}(n t_{1}\ldots t_{n}) \colon \,   
	\Omega_{n} \times S(t_{1}) \times \ldots \times S(t_{n}) \times A \, 
	\longrightarrow \, A, \\
	(\omega,s_{1},\ldots,s_{n},z) \, \longmapsto \, 
	E_{n}(\omega,\phi_{\mathbf{A}}(t_{1})(s_{1},z), \ldots,
	\phi_{\mathbf{A}}(t_{n})(s_{n},z)) \quad
\end{multline*} for all $n \in \mathbb{N}$ and $t_{1},\ldots,t_{n} \in T_{\Sigma}(x)$.  Let $D_{\Omega}(A) \defeq \bigcup \{ \{ t \} \times S(t) \mid t \in T_{\Sigma}(x) \}$ and define $\Phi_{\mathbf{A}} \colon D_{\Omega}(A) \to A^A$ by $\Phi_{\mathbf{A}}(t,s)(z) = \phi_{\mathbf{A}}(t)(s,z)$ for all $(t,s) \in D_{\Omega}(A)$ and $z \in A$. Now the following turns out to be true:

\begin{lem}\label{lemma:bounded.translation.monoid} If $\mathbf{A} = (A,E)$ is an $\Omega$-algebra being affinely bounded by $m \in \mathbb{N}$, then \begin{displaymath}
	M(\mathbf{A}) = \{ \Phi_{\mathbf{A}}(t,s) \mid t \in L_{\Sigma}^{\times}(x), \, \ar(t) \leq m, \, \h(t) \leq m, \, s \in S(t) \} .
\end{displaymath} \end{lem}

Towards proving Lemma~\ref{lemma:bounded.translation.monoid}, we establish mutually inverse maps $\mu_{A} \colon D_{\Omega}(A) \to T_{\Omega + A}(x)$ and $\lambda_{A} \colon T_{\Omega +  A}(x) \to D_{\Omega}(A)$ such that $\Phi_{\mathbf{A}} \circ \lambda_{A} = \Psi_{\mathbf{A}}$. We define $\mu_{A} \colon D_{\Omega}(A) \to T_{\Omega + A}(x)$ by recursion: $\mu_{A}(x,\varnothing) \defeq x$ and $\mu_{A}(\ast, a) \defeq a$ for $a \in A$, as well as 
\begin{displaymath}
	\mu_{A}(n t_{1}\ldots t_{n}, (\omega ,s_{1},\ldots,s_{n})) \defeq \omega \mu_{A}(t_{1},s_{1}) \ldots \mu_{A}(t_{n},s_{n})
\end{displaymath} 
whenever $n \in \mathbb{N}$, $t_{1},\ldots,t_{n} \in T_{\Sigma}(x)$ and $(\omega ,s_{1},\ldots,s_{n}) \in S(n t_{1}\ldots t_{n})$. Conversely, we first define a map $\rho_{A} \colon T_{\Omega + A}(x) \to T_{\Sigma}(x)$ as follows: $\rho_{A}(x) \defeq x$ and $\rho_{A}(a) \defeq \ast$ for $a \in A$, and 
	$\rho_{A}(\omega t_{1} \ldots t_{n}) \defeq n\rho_{A}(t_{1})\ldots \rho_{A}(t_{n})$
for $n \in \mathbb{N}$, $\omega \in \Omega_{n}$ and $t_{1},\ldots ,t_{n} \in T_{\Omega + A}(x)$. Furthermore, for each $t \in T_{\Omega + A}(x)$, we define $\sigma_{A}(t) \in S(\rho_{A}(t))$ by the following recursion: $\sigma_{A}(x) \defeq \varnothing$ and $\sigma_{A}(a) \defeq a$ for $a \in A$, and 
	$\sigma_{A}(\omega t_{1}\ldots t_{n}) \defeq (\omega,\sigma_{A}(t_{1}),\ldots ,\sigma_{A}(t_{n}))$
whenever $n \in \mathbb{N}$, $\omega \in \Omega_{n}$ and $t_{1},\ldots,t_{n} \in T_{\Omega + A}(x)$. Finally, we define $\lambda_{A} \colon T_{\Omega + A}(x) \to D_{\Omega}(A)$, $t \mapsto (\rho_{A}(t),\sigma_{A}(t))$. An elementary term induction provides us with the following observations, which altogether readily imply Lemma~\ref{lemma:bounded.translation.monoid}.

\begin{lem} If $\mathbf{A} = (A,E)$ is an algebra of type $\Omega$, then the following hold: \begin{enumerate}
	\item ${\mu_{A}} \circ {\lambda_{A}} = \id_{T_{\Omega + A}(x)}$ and ${\lambda_{A}} \circ {\mu_{A}} = \id_{D_{\Omega}(A)}$.
	\item $\Phi_{\mathbf{A}} \circ \lambda_{A} = \Psi_{\mathbf{A}}$ and hence $\Psi_{\mathbf{A}} \circ \mu_{A} = \Phi_{\mathbf{A}}$.
	\item $\h \circ {\rho_{A}} = \h$ and $\ar \circ {\rho_{A}} = \ar$.
\end{enumerate} \end{lem}

\section{Affinely bounded topological algebras}\label{section:polynomially.bounded.topological.algebras}

Affine boundedness allows transferring compactness properties of a topological algebra to the respective translation monoid (see Proposition~\ref{proposition:second.main.theorem}). In order to explain this in detail, we need to endow the translation monoid with a suitable topology and then substantiate that the construction of Lemma~\ref{lemma:bounded.translation.monoid} is compatible with the chosen topological setting. 

For this purpose, we first recall some additional topological terminology. Let~$X$ and~$Y$ be topological spaces. We endow the set $C(X,Y)$ of all continuous maps from~$X$ to~$Y$ with the \emph{compact-open topology}, i.e., the topology on $C(X,Y)$ generated by the subbase 
	$\{ [K,U] \mid K \subseteq X \text{ compact}, \, U \subseteq Y \text{ open} \}$,
where $[K,U] \defeq \{ f \in C(X,Y) \mid f(K) \subseteq U \}$ for any compact $K \subseteq X$ and open $U \subseteq Y$. It is easy to see that if $Z$ is another topological space and $f \colon Z \times X \to Y$ is a continuous map, then the map $F \colon Z \to C(X,Y)$ given by \begin{displaymath}
	F(z)(x) \defeq f(z,x) \qquad (z \in Z, \, x \in X)
\end{displaymath} is continuous with respect to the compact-open topology on $C(X,Y)$.

Now we come to topological algebras. Let $\Omega = (\Omega_{n})_{n \in \mathbb{N}}$ be a \emph{continuous signature}, i.e., a sequence of disjoint topological spaces. This subsumes the ordinary notion of signature, since one can regard the latter as a continuous signature in which $\Omega_n$ is discrete for every $n \in \mathbb{N}$. We call $\Omega$ \emph{compact} if $\Omega_{n}$ is compact for every $n \in \mathbb{N}$. A \emph{topological $\Omega$-algebra} or \emph{topological algebra of type $\Omega$} is a pair $\mathbf{A} = (A,E)$ consisting of a topological space~$A$ and a family $E = (E_{n})_{n \in \mathbb{N}}$ of continuous maps $E_{n} \colon \Omega_{n} \times A^{n} \to A$ ($n \in \mathbb{N}$). We adopt all the concepts introduced in Section~\ref{section:polynomial.boundedness} for topological algebras in an obvious manner by referring to the underlying algebra. In particular, note that if $\mathbf{A} = (A,E)$ is a topological $\Omega$-algebra, then $M(\mathbf{A})$ is a subset of~$C(A,A)$. Furthermore, if $A$ is a topological space such that $A \cap \Omega_{n} = \varnothing$ for all $n \in \mathbb{N}$, we obtain a continuous signature $\Omega + A$ by putting $(\Omega + A)_{n} \defeq \Omega_{n}$ for $n \geq 1$ and endowing $(\Omega + A)_{0} \defeq \Omega_{0} \cup A$ with the coproduct topology, i.e., the final topology generated by the inclusion maps $\Omega_{0} \to \Omega_{0} \cup A$ and $A \to \Omega_{0} \cup A$. As in Section~\ref{section:polynomial.boundedness}, we consider the signature $\Sigma \defeq (\Sigma_{n})_{n \in \mathbb{N}}$ where $\Sigma_{0} \defeq \{ 0, \ast \}$ and $\Sigma_{n} \defeq \{ n \}$ for all $n \geq 1$. Now let $\mathbf{A}$ be a topological algebra of type $\Omega$. In accordance with the construction preceding Lemma~\ref{lemma:bounded.translation.monoid}, we define a topological space $S(t)$ for each $t \in T_{\Sigma}(x)$ recursively as follows. We endow $S(x) \defeq \{ \varnothing \}$ with the unique topology on this singleton-set, $S(\ast) \defeq A$ with the given topology on $A$, and \begin{displaymath} 
	S (\omega t_{1} \ldots t_{n}) \defeq \Omega_{n} \times S(t_{1}) \times \ldots \times S(t_{n})
\end{displaymath} with the product topology whenever $n \in \mathbb{N}$, $\omega \in \Sigma_{n}$ and $t_{1},\ldots,t_{n} \in T_{\Sigma}(x)$.  We equip $D_{\Omega}(A)$ with the coproduct topology induced by the spaces $(S(t) \mid t \in T_{\Sigma}(x))$, i.e., the final topology generated by the maps $S(t) \to D_{\Omega}(A), \, s \mapsto (t,s)$ where $t \in T_{\Sigma}(x)$.

\begin{lem}\label{lemma:continuity} Let $\Omega$ be a continuous signature and let $\mathbf{A} = (A,E)$ be a topological algebra of type~$\Omega$. For each $t \in T_{\Sigma}(x)$, the map $\phi_{\mathbf{A}}(t)\colon S(t) \times A \to A$ is continuous. Consequently, the map $\Phi_{\mathbf{A}} \colon D_{\Omega}(A) \to C(A,A)$ is continuous with respect to the compact-open topology. \end{lem}

\begin{proof} A straightforward term induction reveals that $\phi_{\mathbf{A}}(t)\colon S(t) \times A \to A$ is continuous for every $t \in T_{\Sigma}(x)$. Consequently, the map $D_{\Omega}(A) \times A \to A, \, ((t,s),z) \mapsto \phi_{\mathbf{A}}(t)(s,z)$ is continuous. Hence, it follows that $\Phi_{\mathbf{A}} \colon D_{\Omega}(A) \to C(A,A)$ is continuous with respect to the compact-open topology. \end{proof}

\begin{prop}\label{proposition:second.main.theorem} If $\Omega$ is compact and $\mathbf{A}$ is an affinely bounded compact topological $\Omega$-algebra, then $M(\mathbf{A})$ is compact with respect to the compact-open topology. \end{prop}

\begin{proof} Let us first note the following: since $A$ is compact and $\Omega$ is a family of compact spaces, it follows by term induction that $S(t)$ is compact for every $t \in T_{\Sigma}(x)$. Now, suppose $\mathbf{A}$ to be affinely bounded by some $m \in \mathbb{N}$. As $L \defeq \{ t \in L^{\times}_{\Sigma}(x) \mid \ar(t) \leq m, \, \h(t) \leq m \}$ is finite, we conclude that $T \defeq \bigcup \{ \{ t \} \times S(t) \mid t \in L \}$ is a compact subset of $D_{\Omega}(A)$. By Lemma~\ref{lemma:bounded.translation.monoid}, it follows that $M(\mathbf{A}) = \Phi_{\mathbf{A}}(T)$. Since $\Phi_{A} \colon D_{A}(\Omega) \to C(A,A)$ is continuous due to Lemma~\ref{lemma:continuity}, we deduce that $M(\mathbf{A})$ is compact. \end{proof}

\section{Profinite topological algebras}\label{section:profinite.topological.algebras}

In this section we study profinite topological algebras. We show that a topological algebra is profinite if and only if its carrier space is a Stone space and its translation monoid is equicontinuous (Theorem~\ref{theorem:first.main.theorem}). In view of the well-known Arzela-Ascoli theorem, we furthermore reformulate this result in terms of the compact-open topology (Corollary~\ref{corollary:first.main.theorem}). Applying Proposition~\ref{proposition:second.main.theorem}, we conclude that an affinely bounded topological algebra over a compact signature is profinite if and only if its carrier space is a Stone space (Theorem~\ref{theorem:bounded.profinite.algebras}). 

Of course, we need to clarify some terminology. To this end, let $\mathbf{A} = (A,E)$ be a topological $\Omega$-algebra. We say that $\mathbf{A}$ is \emph{residually finite} if, for any two distinct elements $x,y \in A$, there exist a finite discrete topological $\Omega$-algebra $\mathbf{B}$ as well as a continuous homomorphism $\phi \colon \mathbf{A} \to \mathbf{B}$ such that $\phi (x) \ne \phi (y)$. Note that if $\mathbf{A}$ is residually finite, then $A$ is a Hausdorff space. We call $\mathbf{A}$ \emph{profinite} if $A$ is compact and $\mathbf{A}$ is residually finite. We refer to \cite{clark_et_al} for a detailed account on known characterizations of profiniteness, e.g., representability as a projective limit of finite discrete algebras.

In order to state and prove the above-mentioned Theorem~\ref{theorem:first.main.theorem}, we also need to recall some concepts concerning uniform spaces from \cite{Bourbaki1,Bourbaki2}. A \emph{uniformity} on $X$ is a filter $\mathcal{U}$ on the set $X \times X$ such that \begin{enumerate}
	\item $\forall \alpha \in \mathcal{U} \colon \, \Delta_{X} = \{ (x, x) \mid x \in X \} \subseteq \alpha$,
	\item $\forall \alpha \in \mathcal{U} \colon \, \alpha^{-1} \in \mathcal{U}$,
	\item $\forall \alpha \in \mathcal{U} \, \exists \beta \in \mathcal{U} \colon \, \beta \circ \beta \subseteq \alpha$.
\end{enumerate} A \emph{uniform space} is a non-empty set $X$ equipped with a uniformity on $X$, whose elements are called the \emph{entourages} of~$X$. For a uniform space $X$, the \emph{induced topology} on $X$ is defined as follows: a subset $S \subseteq X$ is \emph{open} in $X$ if, for every $x \in S$, there exists an entourage $\alpha$ of~$X$ such that $[x]_{\alpha} \defeq \{ y \in X \mid (x,y) \in \alpha \}$ is contained in $S$. Let $Z$ be a topological space. A set $F \subseteq X^{Z}$ is called \emph{equicontinuous} if for each point $z \in Z$ and every entourage $\alpha$ of $X$ there exists a neighborhood $U$ of $z$ in $Z$ such that \begin{displaymath}
	\forall f \in F \colon \, f(U) \times f(U) \subseteq \alpha .
\end{displaymath} Note that a map $f \colon Z \to X$ is continuous with respect to the induced topology on $X$ if and only if the set $\{ f \}$ is equicontinuous. Now, let $Y$ be another uniform space. A map $f \colon X \to Y$ is called \emph{uniformly continuous} if for every entourage $\alpha$ of~$Y$ there exists some entourage $\beta$ of~$X$ such that $(f \times f)(\beta) \subseteq \alpha$. Furthermore, a set $F \subseteq Y^{X}$ is called \emph{uniformly equicontinuous} if for every entourage $\alpha$ of $Y$ there exists some entourage $\beta$ of~$X$ such that \begin{displaymath}
	\forall f \in F \colon \, (f \times f)(\beta) \subseteq \alpha .
\end{displaymath} It is easy to see that any uniformly continuous map between uniform spaces is continuous with regard to the respective topologies, and that any uniformly equicontinuous set of maps between two uniform spaces is equicontinuous in the sense above.

A crucial example of uniform spaces is provided by the class of compact Hausdorff spaces.

\begin{prop}[{cf.~\cite[\S4.1, Th.~1]{Bourbaki1}, \cite[\S3.1, Cor.~2]{Bourbaki2}}]\label{proposition:compact.hausdorff.spaces} Let $X$ be a compact Hausdorff space. Then the following hold: \begin{enumerate}
	\item $\{ \alpha \subseteq X \times X \mid \exists \beta \subseteq X \times X \textit{ open}\colon \, \Delta_{X} \subseteq \beta \subseteq \alpha \}$ is the unique uniformity on $X$ inducing the topology of $X$.
	\item Let $Y$ be a uniform space. If $F \subseteq Y^{X}$ is equicontinuous, then $F$ is uniformly equicontinuous. In particular, if $f \colon X \to Y$ is continuous, then $f$ is uniformly continuous.
\end{enumerate} \end{prop}

In order to relate profiniteness of a topological algebra to equicontinuity of its translation monoid, we establish the following observation.

\begin{lem}\label{lemma:congruences.generate.uniformity} Let $\Omega$ be a continuous signature and let $\mathbf{A} = (A,E)$ be a profinite topological algebra of type $\Omega$. If $\alpha$ is an entourage of $A$, then there exist a finite discrete $\Omega$-algebra $\mathbf{B}$ and a continuous homomorphism $\phi \colon \mathbf{A} \to \mathbf{B}$ such that $\ker \phi \subseteq \alpha$. \end{lem}

\begin{proof} Let $\beta$ denote the interior of $\alpha$ in the product space $A \times A$. Evidently, $\Delta_{A} \subseteq \beta$ because $\alpha$ is an entourage of $A$. As $\mathbf{A}$ is residually finite, there exist a family of finite discrete topological $\Omega$-algebras $(\mathbf{B}_{i})_{i \in I}$ as well as a corresponding family of continuous homomorphisms $\phi_{i} \colon \mathbf{A} \to \mathbf{B}_{i}$ ($i \in I$) such that \begin{displaymath}
	\forall (x,y) \in (A \times A)\setminus \beta \, \exists i \in I \colon \, \phi_{i}(x) \ne \phi_{i}(y) .
\end{displaymath} It follows that $\{ (\phi_{i} \times \phi_{i})^{-1}((B_{i} \times B_{i})\setminus \Delta_{B_{i}}) \mid i \in I \}$ provides a cover of $(A \times A)\setminus \beta$ by open subsets of $A \times A$. Since $(A \times A)\setminus \beta$ is a closed and hence compact subset of $A \times A$, there exists a finite subset $F \subseteq I$ such that \begin{displaymath}
	(A \times A)\setminus \beta \subseteq \bigcup\nolimits_{i \in F} (\phi_{i} \times \phi_{i})^{-1}((B_{i} \times B_{i})\setminus \Delta_{B_{i}}) ,
\end{displaymath} which means that $\bigcap_{i \in F} \ker \phi_{i} \subseteq \beta$. Now, consider the finite discrete topological $\Omega$-algebra $\mathbf{B} \defeq \prod_{i \in F} \mathbf{B}_{i}$ and the continuous homomorphism $\phi \colon \mathbf{A} \to \mathbf{B}$ defined by $\phi (a) \defeq (\phi_{i}(a))_{i \in F}$ for all $a \in A$. Evidently, $\ker \phi = \bigcap_{i \in F} \ker \phi_{i} \subseteq \beta$. This proves the claim. \end{proof}

The subsequent lemma ensures that Hausdorff quotients of compact topological algebras in fact provide topological algebras.

\begin{lem}\label{lemma:topological.quotient.algebras} Let $\Omega$ be a continuous signature and let $\mathbf{A} = (A,E)$ be a compact topological algebra of type $\Omega$.  Let $\theta$ be a closed congruence of $\mathbf{A}$ and consider the Hausdorff quotient space $B \defeq A/\theta$. Then the map \begin{displaymath}
	E_{n}^{\ast} \colon \Omega_{n} \times B^{n} \to B , \quad (\omega,([a_{1}]_{\theta},\ldots ,[a_{n}]_{\theta})) \mapsto [E_{n}(\omega ,(a_{1},\ldots ,a_{n}))]_{\theta}
\end{displaymath} is continuous for every $n \in \mathbb{N}$. That is, $\mathbf{A}/\theta = (B,E^{\ast})$ is a topological $\Omega$-algebra. \end{lem}

\begin{proof} Recall that~$B$ is a Hausdorff space, since~$A$ is a compact space and~$\theta$ is closed in $A \times A$ (see~\cite[\S 10.4, Prop.~8]{Bourbaki1}). Denote the quotient map by $\phi \colon A \to B, \, a \mapsto [a]_{\theta}$. Let $n \in \mathbb{N}$. We claim that $E_{n}^{\ast} \colon \Omega_{n} \times B^{n} \to B$ is continuous. Evidently, $\psi_{n} \colon A^{n} \to B^{n}$,  $(a_1, \dots, a_n) \mapsto ([a_1]_{\theta}, \dots, [a_n]_{\theta})$ is a continuous map from a compact space to a Hausdorff space. From this it readily follows that $\psi_{n}$ is proper (see \cite[\S 10.2, Cor.~2]{Bourbaki1}). Hence, the continuous surjection $\pi_{n} \colon \Omega_{n} \times A^{n} \to \Omega_{n} \times B^{n}, \, (\omega,a) \mapsto (\omega,\psi_{n}(a))$ is a closed map, and thus $\pi_{n}$ is a quotient map, i.e., the topology of $\Omega_{n} \times B^{n}$ is the final topology generated by~$\pi_{n}$. Furthermore, $E_{n}^{\ast} \circ \pi_{n} = \phi \circ E_{n}$ is continuous. Consequently, $E_{n}^{\ast}$ is continuous. \end{proof}

Next we shall prove the following characterization of profinite topological algebras.

\begin{thm}\label{theorem:first.main.theorem} Let $\Omega$ be a continuous signature and let $\mathbf{A} = (A,E)$ be a topological algebra of type $\Omega$. Then the following are equivalent: \begin{enumerate}
	\item $\mathbf{A}$ is profinite.
	\item $A$ is a Stone space and $M(\mathbf{A})$ is equicontinuous.
\end{enumerate} \end{thm}

\begin{proof} (1)$\Longrightarrow$(2). Suppose $\mathbf{A}$ to be profinite. Of course, this readily implies $A$ to be a compact Hausdorff space. Furthermore, $A$ is totally disconnected as it embeds into a product of discrete and therefore totally disconnected spaces. We are left to show that $M(\mathbf{A})$ is equicontinuous. Let $\alpha$ be an entourage of $A$. Since $\mathbf{A}$ is profinite, Lemma~\ref{lemma:congruences.generate.uniformity} asserts the existence of a finite discrete algebra~$\mathbf{B}$ and a continuous homomorphism $\phi \colon \mathbf{A} \to \mathbf{B}$ such that $\beta \defeq \ker \phi \subseteq \alpha$. By Proposition~\ref{proposition:compact.hausdorff.spaces}, $\phi$ is uniformly continuous. Since $\Delta_{B}$ is an entourage of $B$, this implies that $\beta = (\phi \times \phi)^{-1}(\Delta_{B})$ constitutes an entourage of $A$. Moreover, as $\phi$ is a homomorphism and thus $\beta$ is a congruence, Lemma~\ref{lemma:congruences} asserts that $(f \times f)(\beta) \subseteq \beta \subseteq \alpha$ for every $f \in M(\mathbf{A})$. Thus, $M(\mathbf{A})$ is equicontinuous.

(2)$\Longrightarrow$(1). Let $x,y \in A$ such that $x \ne y$. Since $A$ is a totally disconnected compact Hausdorff space, there exists a clopen subset $C \subseteq A$ such that $C \cap \{ x,y \} = \{ x \}$. Evidently, $\mathcal{P} \defeq \{ C, \, A\setminus C \}$ constitutes a partition of $A$ into clopen subsets. Hence, the corresponding equivalence relation $\alpha \defeq (C \times C) \cup ((A\setminus C) \times (A \setminus C))$ is an open subset of $A \times A$ and therefore an entourage of $A$. According to Lemma~\ref{lemma:congruences}, the equivalence relation $\beta \defeq \bigcap_{f \in M(\mathbf{A})} (f \times f)^{-1}(\alpha )$ is indeed a congruence on $\mathbf{A}$, because $(f(a),f(b)) \in \beta$ for all $(a,b) \in \beta$ and $f \in M(\mathbf{A})$. Now, since $M(\mathbf{A})$ is equicontinuous, there exists some entourage $\gamma$ of $A$ such that $(f \times f)(\gamma) \subseteq \alpha$ for every $f \in M(\mathbf{A})$. Hence, $\gamma \subseteq \beta$. It follows that $A/\beta$ is a collection of disjoint, open subsets of~$A$: in fact, for each $a \in A$, the subset $[a]_{\beta}$ is open in~$A$, as $[b]_{\gamma} \subseteq [b]_{\beta} = [a]_{\beta}$ for every $b \in [a]_{\beta}$. We conclude that the quotient space $A/\beta$ is discrete. As $A$ is compact, this implies $A/\beta$ to be finite. Besides, Lemma~\ref{lemma:topological.quotient.algebras} asserts that $\mathbf{B} \defeq \mathbf{A}/\beta$ is a topological $\Omega$-algebra. Evidently, $\phi \colon A \to A/\beta, \, a \mapsto [a]_{\beta}$ provides a continuous homomorphism from $\mathbf{A}$ to $\mathbf{B}$. Finally, we observe that $\phi(x) \ne \phi(y)$. This shows that $\mathbf{A}$ is profinite. \end{proof}

According to the celebrated Arzela-Ascoli theorem (Theorem~\ref{theorem:arzela.ascoli}), equicontinuity can be described in terms of relative compactness with regard to the corresponding compact-open topology. Recall that a subset $S$ of a topological space $X$ is \emph{relatively compact} if $S$ is contained in a compact subset of $X$.

\begin{thm}[{Arzela-Ascoli, cf.~\cite[\S2.5, Cor.~3]{Bourbaki2}}]\label{theorem:arzela.ascoli} Let~$X$ and~$Y$ be compact Hausdorff topological spaces.  A subset $F \subseteq C(X,Y)$ is equicontinuous if and only if $F$ is relatively compact in $C(X,Y)$. \end{thm}
 	
The subsequent result follows immediately from Theorem~\ref{theorem:first.main.theorem} and Theorem~\ref{theorem:arzela.ascoli}.
 	
\begin{cor}\label{corollary:first.main.theorem} Let $\Omega$ be a continuous signature and let $\mathbf{A} = (A,E)$ be a topological algebra of type $\Omega$. Then the following are equivalent: \begin{enumerate}
	\item $\mathbf{A}$ is profinite.
	\item $A$ is a Stone space and $M(\mathbf{A})$ is relatively compact in $C(A,A)$.
\end{enumerate} \end{cor}

Combining Corollary~\ref{corollary:first.main.theorem} and Proposition~\ref{proposition:second.main.theorem}, we finally obtain the following characterization of affinely bounded profinite algebras over compact signatures.

\begin{thm}\label{theorem:bounded.profinite.algebras} If $\Omega$ is a compact continuous signature and $\mathbf{A} = (A,E)$ is an affinely bounded topological $\Omega$-algebra, then the following are equivalent: \begin{enumerate}
	\item $\mathbf{A}$ is profinite.
	\item $A$ is a Stone space.
\end{enumerate} \end{thm}

\section{Simple topological algebras}\label{section:simple.topological.algebras}

In this section we apply the results established in the previous section in order to show that affinely bounded simple compact algebras satisfy a topological dichotomy -- they are either connected or finite (see Corollary~\ref{corollary:dichotomy}).  As would seem natural, we say that a Hausdorff topological algebra~$\mathbf{A}$ is \emph{simple} if every non-constant continuous homomorphism from~$\mathbf{A}$ into another Hausdorff topological algebra of the same type is injective. We start with a simple reformulation of this property for compact algebras.

\begin{prop}\label{proposition:simple.equivalence} Let~$\Omega$ be any continuous signature. A compact Hausdorff topological algebra~$\mathbf{A} = (A,E)$ of type~$\Omega$ is simple if and only if~$\Delta_{A}$ and $A \times A$ are the only closed congruences on~$\mathbf{A}$. \end{prop}

\begin{proof} ($\Longleftarrow$) Let $\phi \colon \mathbf{A} \to \mathbf{B}$ be a non-constant continuous homomorphism into a Hausdorff topological algebra $\mathbf{B}$ of type $\Omega$. Since $B$ is a Hausdorff space, $\Delta_{B}$ is a closed congruence on $B$. As $\phi$ is a continuous homomorphism, $\ker \phi = \{ (x,y) \in A \times A\mid \phi(x) = \phi(y) \} = (\phi \times \phi)^{-1}(\Delta_{B})$ is a closed congruence on $\mathbf{A}$. By assumption, $\ker \phi = \Delta_{A}$ or $\ker \phi = A \times A$. Since~$\phi$ is non-constant, $\phi$ must be injective. Note that this implication holds in general, i.e., without any compactness assumption.

($\Longrightarrow$) Let~$\theta$ be a closed congruence.  Then Lemma~\ref{lemma:topological.quotient.algebras} readily states that $\mathbf{A}/\theta$ is a Hausdorff topological algebra of type $\Omega$.  Of course, the quotient mapping $\phi \colon \mathbf{A} \to \mathbf{A}/\theta, \, a \mapsto [a]_{\theta}$ is a continuous homomorphism. Hence, simplicity of~$\mathbf{A}$ asserts that~$\phi$ is injective or constant. Consequently, $\theta = \Delta_{A}$ or $\theta = A \times A$. \end{proof}

To clarify some additional notation, let $X$ be a topological space.  For any point $x \in X$, we denote by $C(x)$ the connected component of $x$ in $X$, i.e., the union of all connected subsets of $X$ containing $x$.  Let us remark the following basic observation.

\begin{lem}\label{lemma:connected.relation} If $X$ is a compact Hausdorff topological space, then \begin{displaymath}
	\theta \defeq \{ (x,y) \in X \times X \mid C(x) = C(y) \}
\end{displaymath} is a closed equivalence relation on $X$. \end{lem}

\begin{proof} It is an elementary fact that $\{ C(x) \mid x \in X \}$ constitutes a partition of $X$ into closed, connected subsets.  In particular, $\theta$ is an equivalence relation on~$X$.  To prove that~$\theta$ is closed in $X \times X$, let us recall the following well-known fact (see~\cite[\S 2.4, Prop.~6]{Bourbaki1}): since~$X$ is a compact Hausdorff space, we have \begin{displaymath}
	C(x) = \bigcap \{ C \subseteq X \mid C \text{ clopen}, \, x \in C \}
\end{displaymath} for every $x \in X$. Now, if $(x,y) \in (X \times X)\setminus \theta$, then the equation above asserts the existence of a clopen subset $C \subseteq X$ with $x \in C$ and $y \notin C$, and thus $C \times (X \setminus C)$ is an open neighborhood of $(x,y)$ in $X \times X$ being contained in $(X \times X)\setminus \theta$. Hence, $\theta$ is closed in $X \times X$. \end{proof}

Combining the above with Theorem~\ref{theorem:bounded.profinite.algebras} we arrive at the following result.

\begin{cor}\label{corollary:dichotomy} Let~$\Omega$ be a compact signature and let~$\mathbf{A}$ be an affinely bounded compact Hausdorff topological algebra of type~$\Omega$.  If~$\mathbf{A}$ is simple, then~$\mathbf{A}$ is connected or finite.\end{cor}

\begin{proof} Due to Lemma~\ref{lemma:connected.relation}, the equivalence relation $\theta \defeq \{ (x,y) \in A \times A \mid C(x) = C(y) \}$ is closed in $A \times A$.  Furthermore, it is straightforward to check that $\theta$ is a congruence on~$\mathbf{A}$.  We include a proof for the reader's convenience.  Let $n \in \mathbb{N}$ and $\omega \in \Omega_{n}$. Consider the continuous function $f_{\omega} \colon A^{n} \to A, \, x \mapsto E_{n}(\omega,x)$.  Now let $(a_{1},b_{1}),\ldots,(a_{n},b_{n}) \in \theta$.  Of course, $C(a_{1}) \times \ldots \times C(a_{n})$ is a connected subset of $A^{n}$.  By continuity of $f_{\omega}$, it follows that $f_{\omega}(C(a_{1}) \times \ldots \times C(a_{n}))$ is a connected subset of $A$.  Since $(b_{1},\ldots ,b_{n}) \in C(a_{1}) \times \ldots \times C(a_{n})$ and thus $f_{\omega}(b_{1},\ldots ,b_{n}) \in f_{\omega}(C(a_{1}) \times \ldots \times C(a_{n}))$, we conclude that \begin{displaymath}
	C(f_{\omega}(b_{1},\ldots ,b_{n})) = C(f_{\omega}(a_{1},\ldots ,a_{n})) ,
\end{displaymath} i.e., $(E_{n}(\omega,a_{1},\ldots ,a_{n}),E_{n}(\omega,b_{1},\ldots ,b_{n})) \in \theta$.  This shows that $\theta$ is a congruence on~$\mathbf{A}$.  Consequently, by Proposition~\ref{proposition:simple.equivalence}, simplicity of~$\mathbf{A}$ implies that $\theta = A \times A$ or $\theta = \Delta_{A}$.  The former just means that $A$ is connected.  Whereas in the latter case, $A$ is totally disconnected, and therefore~$\mathbf{A}$ is profinite by Theorem~\ref{theorem:bounded.profinite.algebras} and hence finite by simplicity again. \end{proof}

\section{Examples}\label{section:examples}

This section shall be concerned with examples of topological algebras that are not profinite despite their underlying topological spaces being Stone spaces.  The constructions we present are based on two different compactifications of a countable discrete space equipped with compatible unary operations.  The first example is taken from \cite{banaschewski-pua}.

\begin{exmpl}\label{example:stone.cech} Let $A = \beta \mathbb{N}$ be the Stone-\v{C}ech compactification of the natural numbers~$\mathbb{N}$ and let $f \colon A \to A$ be the continuous extension of the successor function $\mathbb{N} \to \mathbb{N}, \, n \mapsto n + 1$. We consider the continuous signature $\Omega$ where $\Omega_{1} \defeq \{ \varnothing \}$ and $\Omega_{n} \defeq \varnothing$ for all $n \in \mathbb{N}\setminus \{ 1 \}$. We obtain a topological $\Omega$-algebra $\mathbf{A} = (A,E)$ where $E_{1} \colon \{ \varnothing \} \times A \to A, \, (\varnothing,z) \mapsto f(z)$. Of course, $A$ is a totally disconnected compact Hausdorff space. However, $\mathbf{A}$ is not profinite, since $A$ is not metrizable and $\mathbf{A}$ admits only countably many open congruences, i.e., congruences of $\mathbf{A}$ being open in $A \times A$. To see the latter, consider the topological subalgebra $\mathbf{N}$ of $\mathbf{A}$ induced on the subuniverse $\mathbb{N}$. We argue that the map \begin{displaymath}
	\kappa \colon \{ \text{open congruences of } \mathbf{A} \} \to \{ \text{congruences of } \mathbf{N} \}, \quad \theta \mapsto \theta \cap (\mathbb{N} \times \mathbb{N})
\end{displaymath} is injective. For this purpose, let~$\theta_{0}$ and~$\theta_{1}$ be open congruences of~$\mathbf{A}$ with $\theta_{0} \nsubseteq \theta_{1}$. Since~$\theta_{1}$ is an open equivalence relation, $\theta_{1}$ is also closed in $A \times A$, and thus $\theta_{0}\setminus \theta_{1}$ is a nonempty open subset of $A \times A$. Therefore, density of $\mathbb{N}$ in $A$ implies that $(\theta_{0}\setminus \theta_{1}) \cap (\mathbb{N} \times \mathbb{N}) \ne \varnothing$ and hence $\kappa (\theta_{0}) \ne \kappa(\theta_{1})$. Furthermore, we observe that $\mathbf{N}$ has only countably many congruences: if $\theta$ is a congruence of $\mathbf{N}$, then either $\theta = \mathbb{N} \times \mathbb{N}$, or there exist $m,n \in \mathbb{N}$ with $m \ne n$ such that $\theta$ contains the congruence $\theta(m,n)$ of $\mathbf{N}$ generated by the pair $(m,n)$. Since $\mathbb{N}/\theta (m,n)$ is finite for any two distinct $m,n \in \mathbb{N}$, it follows that $\mathbf{N}$ has only countably many congruences. Hence, $\mathbf{A}$ admits only countably many open congruences. Since $\mathbf{A}$ is therefore not profinite, Theorem~\ref{theorem:first.main.theorem} implies that $M(\mathbf{A})$ is not equicontinuous. Indeed, as~$\mathbf{A}$ is not affinely bounded, Proposition~\ref{proposition:second.main.theorem} does not apply.

Let us note an observation concerning this example for a different choice of a signature. With regard to the continuous signature $\Omega^{\ast}$ with $\Omega^{\ast}_{1} \defeq \mathbb{N}$ and $\Omega^{\ast}_{n} \defeq \varnothing$ for $n \in \mathbb{N}\setminus \{ 1 \}$, we obtain a topological $\Omega^{\ast}$-algebra $\mathbf{B} = (A,E^{\ast})$ where $E_{1}^{\ast} \colon \mathbb{N} \times A \to A, \, (n,z) \mapsto f^{n}(z)$. Evidently, $M(\mathbf{B}) = M(\mathbf{A})$ and $\mathbf{B}$ is affinely bounded by $1$. However, Proposition~\ref{proposition:second.main.theorem} does not apply to $\mathbf{B}$ as the continuous signature $\Omega^{\ast}$ is not compact. \end{exmpl}

The subsequent example, which may be found in \cite{clark_et_al}, shows that even under the additional hypothesis of metrizability Stone topological algebras need not be profinite.

\begin{exmpl} Let $A = \alpha \mathbb{Z}$ be the Alexandroff compactification of the integers~$\mathbb{Z}$, i.e., the set $\mathbb{Z}_{\infty} = \mathbb{Z} \cup \{ \infty \}$ equipped with the topology\begin{equation*}
	\{ U \subseteq \mathbb{Z} \cup \{ \infty\} \mid \infty \in U \Rightarrow \mathbb{Z} \setminus U \text{ finite} \} .
\end{equation*} Consider the homeomorphism $f \colon A \to A$ given by $f(\infty ) = \infty$ and $f(z) \defeq z+1$ for $z \in \mathbb{Z}$. Again, we have a continuous signature $\Omega$, where $\Omega_{1} \defeq \{ \varnothing \}$ and $\Omega_{n} \defeq \varnothing$ for all $n \in \mathbb{N}\setminus \{ 1 \}$, and we get a topological $\Omega$-algebra $\mathbf{A} = (A,E)$ where $E_{1} \colon \{ \varnothing \} \times A \to A, \, (\varnothing,z) \mapsto f(z)$. Of course, $A$ is a metrizable, totally disconnected compact Hausdorff space. However, $\mathbf{A}$ is not residually finite, as every continuous homomorphism from $\mathbf{A}$ into a finite discrete $\Omega$-algebra is constant. Indeed, suppose $\phi \colon \mathbf{A} \to \mathbf{B}$ to be a continuous homomorphism into a finite discrete $\Omega$-algebra $\mathbf{B} = (B,E^{\ast})$. By continuity of $\phi$, there exists some $c \in \mathbb{Z}$ such that $\phi (x) = \phi (\infty)$ for all $x \in \mathbb{Z}$ with $x \leq c$. We show that $\phi (z) = \phi (\infty)$ for all $z \in \mathbb{Z}$. To this end, let $z \in \mathbb{Z}$. There exists $k \in \mathbb{N}$ such that $f^{-k}(z) \leq c$. Considering the function $g \colon B \to B, \, z \mapsto E^{\ast}(\varnothing,z)$, we conclude that \begin{displaymath}
	\phi (z) = \phi (f^{k}(f^{-k}(z))) = g^{k}(\phi (f^{-k}(z))) = g^{k}(\phi (\infty)) = \phi (f^{k}(\infty)) = \phi (\infty) .
\end{displaymath} Hence, $\phi$ is constant. It follows that $\mathbf{A}$ is not residually finite and therefore not profinite. Accordingly, Theorem~\ref{theorem:first.main.theorem} implies that $M(\mathbf{A})$ is not equicontinuous. Note again that Proposition~\ref{proposition:second.main.theorem} is not applicable, since $\mathbf{A}$ is not affinely bounded.

Similarly to Example~\ref{example:stone.cech}, when considering the continuous signature $\Omega^{\ast}$ given by $\Omega^{\ast}_{1} \defeq \mathbb{N}$ and $\Omega^{\ast}_{n} \defeq \varnothing$ for $n \in \mathbb{N}\setminus \{ 1 \}$, we obtain a topological $\Omega^{\ast}$-algebra $\mathbf{B} = (A,E^{\ast})$ by setting $E_{1}^{\ast} \colon \mathbb{N} \times A \to A, \, (n,z) \mapsto f^{n}(z)$. Likewise, $M(\mathbf{B}) = M(\mathbf{A})$ and $\mathbf{B}$ is affinely bounded by $1$. Again, Proposition~\ref{proposition:second.main.theorem} does not apply to $\mathbf{B}$ as the continuous signature $\Omega^{\ast}$ is not compact. \end{exmpl}

Finally, by modifying the latter example we present an infinite, totally disconnected, simple compact Hausdorff topological algebra.  This in particular illustrates that in Corollary~\ref{corollary:dichotomy} the assumption of affine boundedness cannot be dropped in general.

\begin{exmpl} Let $A = \alpha \mathbb{Z}$ be again the Alexandroff compactification of the integers~$\mathbb{Z}$. Consider the homeomorphism $f \colon A \to A$ given by $f(\infty ) = \infty$ and $f(z) \defeq z+1$ for $z \in \mathbb{Z}$. We obtain a continuous signature $\Omega$, where $\Omega_{1} \defeq \{ -1, 1 \}$ and $\Omega_{n} \defeq \varnothing$ for all $n \in \mathbb{N}\setminus \{ 1 \}$, and we get a topological $\Omega$-algebra $\mathbf{A} = (A,E)$ where $E_{1} \colon \{ -1, 1 \} \times A \to A, \, (\epsilon,z) \mapsto f^{\epsilon}(z)$. As before, $A$ is an infinite, totally disconnected, compact Hausdorff space. Moreover, $\mathbf{A}$ is simple. We shall prove this by applying Proposition~\ref{proposition:simple.equivalence}. Let $\theta$ be a closed congruence on $\mathbf{A}$. Suppose that $\theta \ne \Delta_{A}$. We show that $\theta = A \times A$.

First we prove that there exists $x \in \mathbb{Z}$ such that $(x,\infty) \in \theta$. By assumption, there exists a pair $(x,y) \in \theta \setminus \Delta_{A}$. Clearly, if $\infty \in \{ x,y \}$, then we are done with proving our claim. So, assume that $\{ x,y \} \subseteq \mathbb{Z}$. Consider $k \defeq y-x \in \mathbb{Z}\setminus \{ 0 \}$. By induction, it follows that $(f^{nk}(x),f^{(n+1)k}(x))$ for all $n \in \mathbb{N}$. Hence, $(x,f^{nk}(x)) \in \theta$ for all $n \in \mathbb{N}$. Since $\theta$ is closed in $A \times A$, we conclude that $(x,\infty) \in \theta$.

Now we prove that $\theta = A \times A$. As shown above, there is $x \in \mathbb{Z}$ with $(x,\infty) \in \theta$. It follows that $(z,\infty) = (f^{z-x}(x),f^{z-x}(\infty)) \in \theta$ for every $z \in \mathbb{Z}$. Thus, $\theta = A \times A$. This shows that $\mathbf{A}$ is simple. Proposition~\ref{proposition:second.main.theorem} does not apply, as $\mathbf{A}$ is not affinely bounded. \end{exmpl}

\section{Applications}\label{section:applications}

In Example~\ref{example:polynomial.boundedness} we have seen that groups and semigroups are instances of affinely bounded algebras (where the signature is finite).  We therefore re-establish from Theorem~\ref{theorem:bounded.profinite.algebras} the classical equivalence theorems by van Dantzig~\cite{VanDantzig} and Numakura~\cite{Numakura57}, namely that a topological group or topological semigroup, respectively, is profinite if and only if the underlying topological space is a totally disconnected compact Hausdorff space.  Rings and distributive lattices are easily seen to be affinely bounded as well (cf.~Example~\ref{example:semirings}), so that we also recover the corresponding equivalence theorems by Anzai~\cite{anzai} and Numakura~\cite{Numakura57}.

In this section we shall establish further examples for affinely bounded topological algebras, so that Theorem~\ref{theorem:bounded.profinite.algebras} can be applied. This will lead to the characterization of profiniteness for topological semirings (Proposition~\ref{proposition:profinite.semirings}), modules over compact semirings (Proposition~\ref{proposition:profinite.modules}), as well as distributive associative algebras in the sense of Choe~\cite{choe} (Corollary~\ref{corollary:profinite.choe.algebras}).

\begin{exmpl}\label{example:semirings} Let $\Omega = ( \Omega_n )_{n \in {\mathbb{N}}}$ with $\Omega_2 = \{ +, \cdot \}$ and $\Omega_i = \varnothing$ for $i \ne 2$, and let $\mathbf{R} = (R, +, \cdot)$ be a \emph{semiring}, i.e., $\mathbf{R}$ is an $\Omega$-algebra such that the binary operations~$+$ and~$\cdot$ on~$R$ are associative, whereas $+$ is commutative, and the distributive laws $x \cdot (y + z) = x \cdot y + x \cdot z$ and $(x + y) \cdot z = x \cdot z + y \cdot z$ hold for all $x, y, z \in R$.  Then the $\Omega$-algebra~$\mathbf{R}$ is affinely bounded by~$3$.  Indeed, for each affine $\Omega$-term $t \in L_{\Omega + R}(x)$ over~$R$ there exists $t' \in L_{\Omega + R}(x)$ with $\Psi_{\mathbf{R}}(t) = \Psi_{\mathbf{R}}(t')$ and $\h(t') \le 3$, in fact one can choose $t' = ((a \cdot x) \cdot b) + c$ (or $t' = (a \cdot x) + b$, $t' = (x \cdot a) + b$, $t' = x + a$, $t' = (a \cdot x) \cdot b$, $t' = a \cdot x$, $t' = x \cdot a$, or $t' = x$) for some $a, b, c \in R$. 

Analogously, one observes that if $\mathbf{R} = (R, +, \cdot, 0)$ is a ring, then the algebra~$\mathbf{R}$ is likewise affinely bounded by~$3$.
\end{exmpl}

Theorem~\ref{theorem:bounded.profinite.algebras} and Corollary~\ref{corollary:dichotomy} applied to Example~\ref{example:semirings} readily provide us with the following.

\begin{prop}\label{proposition:profinite.semirings} For any topological semiring $\mathbf{R} = (R, +, \cdot)$ the following hold: \begin{enumerate}
	\item $\mathbf{R}$ is profinite if and only if $R$ is a Stone space.
	\item If $\mathbf{R}$ is compact and simple, then $R$ is connected or finite.
\end{enumerate} \end{prop}

Our next example is provided by the class of Boolean algebras.

\begin{exmpl} Let $\mathbf{A} = (A, \vee, \wedge, \neg, 0, 1)$ be a {\em Boolean algebra}, i.e., an $\Omega$-algebra, where $\Omega_0 = \{ 0, 1 \}$, $\Omega_1 = \{ \neg \}$, $\Omega_2 = \{ \vee, \wedge \}$, $\Omega_i = \varnothing$ for $i > 2$, and such that the binary operations~$\vee$ and~$\wedge$ are associative and commutative, are intertwined by the absorption laws and the distributive laws, and $a \vee 0 = a = a \wedge 1$, $a \vee (\neg a) = 1$, $a \wedge (\neg a) = 0$ hold for all $a \in A$.  Then the $\Omega$-algebra~$\mathbf{A}$ is affinely bounded by~$3$, since for every affine $\Omega$-term $t \in L_{\Omega+A}(x)$ over~$A$ there exists $t' \in L_{\Omega+A}(x)$ with $\Psi_{\mathbf{A}}(t) = \Psi_{\mathbf{A}}(t')$ and $\h(t') \le 3$, namely $t' = (x \vee a) \wedge b$ or $t' = ((\neg x) \vee a) \wedge b$ for some $a, b \in A$. \end{exmpl}

As a consequence we infer from Theorem~\ref{theorem:bounded.profinite.algebras} that a topological Boolean algebra is profinite if and only if its underlying space is a Stone space.  Note that this result also follows already from the characterization of profinite distributive lattices, since every lattice congruence of a Boolean algebra is in fact a congruence of the algebra.

The following class of examples illustrates in particular the benefit of considering compact continuous signatures in general instead of just finite ones.

\begin{exmpl}\label{example:semimodules} Let $\mathbf{R} = (R, +, \cdot)$ be a semiring and let $\Omega = ( \Omega_n )_{n \in \mathbb{N}}$ be the signature with $\Omega_1 = R$, $\Omega_2 = \{ + \}$, and $\Omega_i = \varnothing$ for $i \in \mathbb{N} \setminus \{ 1, 2 \}$.  A {\em semimodule} over~$\mathbf{R}$ is an $\Omega$-algebra $\mathbf{M} = (M, T, +)$, where $+$ is a commutative, associative binary operation on~$M$, and $T: R \times M \to M$, $(r, m) \mapsto T_r( m )$, is a map such that $T_r(x + y) = T_r(x) + T_r(y)$, $T_{r+s}(x) = T_r(x) + T_s(x)$, and $T_{r \cdot s}(x) = T_r( T_s( x ))$ hold for all $r, s \in R$ and $x, y \in M$.  The $\Omega$-algebra~$\mathbf{M}$ is affinely bounded by~$2$, since for each $t \in L_{\Omega + M}(x)$ there exists $t' \in L_{\Omega + M}(x)$ with $\Psi_{\mathbf{M}}(t) = \Psi_{\mathbf{M}}(t')$ and $\h(t') \le 2$, namely $t' = T_r(x) + a$ (or $t' = x + a$, $t' = T_r(x)$, or $t' = x$) for some $r \in R$ and $a \in M$. \end{exmpl}

When Theorem~\ref{theorem:bounded.profinite.algebras} and Corollary~\ref{corollary:dichotomy} are applied now to Example~\ref{example:semimodules}, we obtain our next result.

\begin{prop}\label{proposition:profinite.modules} Let $\mathbf{R} = (R, +, \cdot)$ be a compact topological semiring and let $\mathbf{M} = (M, T, +)$ be a topological semimodule over~$\mathbf{R}$.  Then the following hold: \begin{enumerate}
	\item $\mathbf{M}$ is profinite if and only if $M$ is a Stone space.
	\item If $\mathbf{M}$ is compact and simple, then $M$ is connected or finite.
\end{enumerate} \end{prop}

We conclude this article with a more general class of affinely bounded algebras. First we shall briefly discuss associative algebras. For this purpose, we recall some terminology from \cite{choe}. Let $A$ be a set. A function $f \colon A \to A$ is called \emph{associative} if the transformation monoid generated by $f$ is finite, i.e., there exist $k \in \mathbb{N}$ and $l \in \mathbb{N}\setminus \{ 0 \}$ such that $f^{k+l} = f^{k}$. Furthermore, let $n \in \mathbb{N}\setminus \{ 0, 1 \}$. A function $f \colon A^{n} \to A$ is is said to be \emph{associative} if \begin{multline*}
	f(f(a_{1},\ldots,a_{n}),a_{n+1},\ldots,a_{2n-1}) = f(a_{1},f(a_{2},\ldots,a_{n+1}),a_{n+2},\ldots ,a_{2n-1}) \\
	= \ldots = f(a_{1},\ldots,a_{n-1},f(a_{n},\ldots,a_{2n-1}))
\end{multline*} for all $a_{1},\ldots,a_{2n-1} \in A$. Given a signature $\Omega$, we call an $\Omega$-algebra $\mathbf{A} = (A,E)$ \emph{associative} if the function $\omega^{\mathbf{A}} \colon A^{n} \to A$ is associative for every $\omega \in \Omega_{n}$ and $n \in \mathbb{N}\setminus \{ 0 \}$.

As the following example shows, associativity of an operation significantly affects the representation of the respective translation monoid.

\begin{exmpl}\label{example:associative.algebras} Let us fix some $n \in \mathbb{N} \setminus \{ 0 \}$. We define $\Omega \defeq (\Omega_{m})_{m \in \mathbb{N}}$ by $\Omega_{n} \defeq \{ \varnothing\}$ and $\Omega_{m} \defeq \varnothing $ for $m \in \mathbb{N}\setminus \{ n \}$. Let $\mathbf{A} = (A,E)$ be an associative algebra of type $\Omega$. We show that~$\mathbf{A}$ is affinely bounded. To this end, consider the function $f \colon A^{n} \to A, \, z \mapsto E(\varnothing,z)$. If $n= 1$, then $M(\mathbf{A})$ is finite and $\mathbf{A}$ is affinely bounded by $\vert M(\mathbf{A}) \vert -1$. Otherwise, if $n > 1$, then a straightforward induction reveals that \begin{multline*}
	M(\mathbf{A}) = \{ \id_{A} \} \cup \{ x \mapsto f(a,x,b) \mid a \in A^{k}, \, b \in A^{n-k-1}, \, k\in \{ 0,\ldots ,n-1 \} \} \\
	\cup \{ x \mapsto f(f(a,x),b) \mid a,b \in A^{n-1} \}
\end{multline*} (compare with Example~\ref{example:polynomial.boundedness}). In particular, it follows that $\mathbf{A}$ is affinely bounded by $2$. \end{exmpl}

Our next purpose is to extend the previous observation to distributive associative algebras. To this end, we shall first point out a useful general fact concerning translation monoids. Again we first address some notational issues. Let $\Omega$ be a signature and let $\mathbf{A} = (A,E)$ be an $\Omega$-algebra. If $\Omega' = (\Omega'_{n})_{n \in \mathbb{N}}$ is a family of subsets $\Omega'_{n} \subseteq \Omega_{n}$ ($n \in \mathbb{N}$), then $\Omega'$ itself constitutes a signature and we obtain an $\Omega'$-algebra $\mathbf{A}_{\Omega'} \defeq (A,E_{\Omega'})$ in a natural manner by restriction. Considering a subset $\Omega_{n}' \subseteq \Omega_{n}$ with $n \in \mathbb{N}$, we define $\mathbf{A}_{\Omega_{n}'} \defeq \mathbf{A}_{\Omega'}$ where $\Omega' \defeq (\Omega'_{n})_{n \in \mathbb{N}}$ with $\Omega'_{m} \defeq \varnothing$ for $m \in \mathbb{N}\setminus \{ n \}$. In particular, we define $\mathbf{A}_{\omega} \defeq \mathbf{A}_{\{ \omega \}}$ for a single symbol $\omega \in \Omega_{n}$ with $n \in \mathbb{N}$. Moreover, let $\Sigma$ be another signature. Assume that $\Omega$ and $\Sigma$ are disjoint, that is, $\Omega_{m} \cap \Sigma_{n} = \varnothing$ for all $m,n \in \mathbb{N}$. Then we shall consider the signature $\Omega + \Sigma \defeq ((\Omega + \Sigma)_{n})_{n \in \mathbb{N}}$ given by $(\Omega + \Sigma)_{n} \defeq \Omega_{n} \cup \Sigma_{n}$ for $n \in \mathbb{N}$.

With this notation, we note the following observation.

\begin{lem}\label{lemma:commuting.translations.monoids} Let $\Omega$ and $\Sigma$ be disjoint signatures. Let $\mathbf{A} = (A,E)$ be an $(\Omega + \Sigma)$-algebra such that $M(\mathbf{A}) = \{ f \circ g \mid f \in M(\mathbf{A}_{\Omega}), \, g \in M(\mathbf{A}_{\Sigma}) \}$. If $\mathbf{A}_{\Omega}$ is affinely bounded by $k \in \mathbb{N}$ and $\mathbf{A}_{\Sigma}$ is affinely bounded by $l \in \mathbb{N}$, then $\mathbf{A}$ is affinely bounded by $k+l$. \end{lem}

\begin{proof} For the sake of brevity, let us define $S \defeq \{ s \in L_{\Omega + A}^{\times}(x) \mid \h (s) \leq k, \, \ar (s) \leq k \}$ and $T \defeq \{ t \in L_{\Sigma + A}^{\times}(x) \mid \h (t) \leq l, \, \ar (t) \leq l \}$. Since $\Psi_{\mathbf{A}} \colon (L_{\Omega + \Sigma +A}^{\times}(x), \cdot) \to (M(\mathbf{A}), \circ)$ is a monoid homomorphism, we conclude that \begin{align*}
	M(\mathbf{A}) &= \{ f \circ g \mid f \in M(\mathbf{A}_{\Omega}), \, g \in M(\mathbf{A}_{\Sigma}) \} = \{ \Psi_{\mathbf{A}_{\Omega}}(s) \circ \Psi_{\mathbf{A}_{\Sigma}}(t) \mid s \in S, \, t \in T \} \\
	&= \{ \Psi_{\mathbf{A}}(s) \circ \Psi_{\mathbf{A}}(t) \mid s \in S, \, t \in T \} = \{ \Psi_{\mathbf{A}}(s \cdot t) \mid s \in S, \, t \in T \} \\
	&\subseteq \{ \Psi_{\mathbf{A}}(t) \mid t \in L_{\Omega + \Sigma + A}^{\times}(x), \, \h (t) \leq k+l , \, \ar (t) \leq \max (k,l) \} ,
\end{align*} which implies that $\mathbf{A}$ is affinely bounded by $k+l$. \end{proof}

Now we come to distributive algebras. For this purpose, we need to recall some additional terminology from \cite{choe}. Let $A$ be a set and let $f \colon A^{n} \to A$ be a function where $n \in \mathbb{N} \setminus \{ 0 \}$. We say that a function $h \colon A \to A$ \emph{distributes} over $f$ if \begin{displaymath}
	h(f(a_{1},\ldots,a_{n})) = f(h(a_{1}),\ldots,h(a_{n}))
\end{displaymath} for all $a_{1},\ldots,a_{n} \in A$, or if there exists $k \in \{ 1,\ldots,n \}$ such that \begin{displaymath}
	h(f(a_{1},\ldots,a_{n})) = f(a_{1},\ldots,a_{k-1},h(a_{k}),a_{k+1},\ldots,a_{n})
\end{displaymath} for all $a_{1},\ldots,a_{n} \in A$. Note that if $n = 1$, then~$h$ distributes over~$f$ if and only if~$h$ and~$f$ commute. Furthermore, a function $g \colon A^{m} \to A$ where $m \in \mathbb{N}\setminus \{ 0,1 \}$ \emph{distributes} over~$f$ if \begin{multline*}
	g(b_{1},\ldots,b_{k-1},f(a_{1},\ldots,a_{n}),b_{k+1},\ldots,b_{m}) \\
	= f(g(b_{1},\ldots,b_{k-1},a_{1},b_{k+1},\ldots,b_{m}), \ldots, g(b_{1},\ldots,b_{k-1},a_{n},b_{k+1},\ldots,b_{m}))
\end{multline*} whenever $k \in \{ 1,\ldots,m \}$, $b_{1},\ldots,b_{k-1},b_{k+1},\ldots,b_{m} \in A$, and $a_{1},\ldots,a_{n} \in A$.

From the definitions above one easily deduces the subsequent lemma.

\begin{lem}\label{lemma:distributive.algebras} Let $\Sigma$ and $\Pi$ be disjoint signatures and let $\mathbf{A} = (A,E)$ be a $(\Sigma + \Pi)$-algebra. Suppose that, for all $\sigma \in \Sigma_{n}$ and $\pi \in \Pi_{m}$ with $n, m \geq 1$, the function $\pi^{\mathbf{A}}$ distributes over the function $\sigma^{\mathbf{A}}$. Then \begin{displaymath}
	M(\mathbf{A}) = \{ f \circ g \mid f \in M(\mathbf{A}_{\Sigma}), \, g \in M(\mathbf{A}_{\Pi}) \} .
\end{displaymath} \end{lem}

\begin{proof} By assumption, we have the following: for every translation~$f$ of $\mathbf{A}_{\Sigma}$ and every translation~$g$ of $\mathbf{A}_{\Pi}$, there is a translation~$\tilde{f}$ of $\mathbf{A}_{\Sigma}$ such that $g \circ f = \tilde{f} \circ g$ or $g \circ f = \tilde{f}$. We conclude that for all $f \in M(\mathbf{A}_{\Sigma})$ and $g \in M(\mathbf{A}_{\Pi})$ there exists $\tilde{f} \in M(\mathbf{A}_{\Sigma})$ and $\tilde{g} \in M(\mathbf{A}_{\Pi})$ such that $g \circ f = \tilde{f} \circ \tilde{g}$. This readily implies the desired equation. \end{proof}



Utilizing the previous two lemmata, we finally re-establish a result of \cite{choe}. For this purpose, let us briefly agree on some additional terminology. By a \emph{Choe-signature} we mean a pair $(\Omega,{\preceq})$ consisting of a signature $\Omega$ where $\bigcup \{ \Omega_{n} \mid n \in \mathbb{N} \}$ is finite and a linear order $\preceq$ on the set $\bigcup \{ \Omega_{n} \mid n \in \mathbb{N}\setminus \{ 0,1 \} \}$. Given a Choe-signature $(\Omega,\preceq)$, an $\Omega$-algebra $\mathbf{A} = (A,E)$ is said to be \emph{distributive} with respect to $\preceq$ if each of the following conditions is satisfied: \begin{enumerate}
	\item if $\sigma \in \Omega_{n}$, $\pi \in \Omega_{m}$, $n,m \geq 1$, and $\sigma \prec \pi$, then $\pi^{\mathbf{A}}$ distributes over $\sigma^{\mathbf{A}}$.
	\item if $\sigma \in \Omega_{n}$, $n \geq 1$, and $\omega \in \Omega_{1}$, then $\omega^{\mathbf{A}}$ distributes over $\sigma^{\mathbf{A}}$.
\end{enumerate} The subsequent result enables us to deduce the main result of \cite{choe} (statement (1) of our Corollary~\ref{corollary:profinite.choe.algebras}) as an immediate consequence of Theorem~\ref{theorem:bounded.profinite.algebras}.

\begin{prop}\label{proposition:associative.distributive.algebras} Let $(\Omega,{\preceq})$ be a Choe-signature. Suppose the $\Omega$-algebra $\mathbf{A} = (A,E)$ to be associative and distributive with respect to $\preceq$. Then $\mathbf{A}$ is affinely bounded by \begin{displaymath}
	2 \sum_{n \geq 2} \vert \Omega_{n} \vert + \sum_{\omega \in \Omega_{1}} \vert M(\mathbf{A}_{\omega}) \vert - \vert \Omega_{1} \vert .
\end{displaymath} \end{prop}

\begin{proof} Let $\Sigma \defeq (\Sigma_{n})_{n \in \mathbb{N}}$ where $\Sigma_{1} \defeq \Sigma_{0} \defeq \varnothing$ and $\Sigma_{n} \defeq \Omega_{n}$ for $n \geq 2$. Since $\omega^{\mathbf{A}}$ distributes over $\sigma^{\mathbf{A}}$ for each $\omega \in \Omega_1$ and each $\sigma \in \Sigma$, by applying Lemma~\ref{lemma:distributive.algebras} we observe that \begin{align*}
	M(\mathbf{A}) = \{ f \circ g \mid f \in M(\mathbf{A}_{\Sigma}), \, g \in M(\mathbf{A}_{\Omega_{1}}) \} .
\end{align*} 
Let us define $\Delta \defeq \bigcup \{ \Omega_{n} \mid n \geq 2 \}$ and $t \defeq \vert \Delta \vert$. As $(\Delta, \preceq)$ is linearly ordered, there is an order-isomorphism $\kappa \colon (\{ 1,\ldots,t \},{\leq}) \to (\Delta,{\preceq})$. Since $\mathbf{A}$ is distributive with respect to~$\preceq$, Lemma~\ref{lemma:distributive.algebras} asserts that \begin{displaymath}
	M(\mathbf{A}_{\Sigma}) = \{ f_{1} \circ \ldots \circ f_{t} \mid f_{1} \in M(\mathbf{A}_{\kappa (1)}), \ldots, f_{t} \in M(\mathbf{A}_{\kappa (t)}) \} .
\end{displaymath} Hence, $\mathbf{A}_{\Sigma}$ is affinely bounded by $2 t = 2 \sum_{n \geq 2} \vert \Omega_{n} \vert$ due to Lemma~\ref{lemma:commuting.translations.monoids} and the second part of Example~\ref{example:associative.algebras}. Furthermore, letting $s \defeq \vert \Omega_{1} \vert$ and considering any bijective map $\nu \colon \{ 1,\ldots,s \} \to \Omega_{1}$, we utilize Lemma~\ref{lemma:distributive.algebras} to see that \begin{displaymath}
	M(\mathbf{A}_{\Omega_{1}}) = \{ f_{1} \circ \ldots \circ f_{s} \mid f_{1} \in M(\mathbf{A}_{\nu (1)}), \ldots, f_{s} \in M(\mathbf{A}_{\nu (s)}) \} .
\end{displaymath} Lemma~\ref{lemma:commuting.translations.monoids} together with the first part of Example~\ref{example:associative.algebras} yields that $\mathbf{A}_{\Omega_{1}}$ is affinely bounded by $\sum_{\omega \in \Omega_{1}} (\vert M(\mathbf{A}_{\omega}) \vert - 1) = \sum_{\omega \in \Omega_{1}} \vert M(\mathbf{A}_{\omega}) \vert - s$. Now, our claim follows by Lemma~\ref{lemma:commuting.translations.monoids}. \end{proof}

As mentioned above, Proposition~\ref{proposition:associative.distributive.algebras} lays the frame for another interesting application of Theorem~\ref{theorem:bounded.profinite.algebras} and Corollary~\ref{corollary:dichotomy}, which constitutes our final result. 

\begin{cor}\label{corollary:profinite.choe.algebras} Let $(\Omega,{\preceq})$ be a Choe-signature. If $\mathbf{A} = (A,E)$ is an associative topological $\Omega$-algebra being distributive with respect to $\preceq$, then the following hold: \begin{enumerate}
	\item (\cite{choe}) $\mathbf{A}$ is profinite if and only if $A$ is a Stone space.
	\item If $\mathbf{A}$ compact and simple, then $A$ is connected or finite.
\end{enumerate} \end{cor}

\section*{Acknowledgments}

F.~M.~Schneider is supported by funding of the Excellence Initiative by the German Federal and State Governments.  J.~Zumbr\"agel has been funded by the Irish Research Council under grant no.~ELEVATEPD/2013/82.

\printbibliography

\end{document}